\documentclass[12pt]{article}
\usepackage{amsmath}
\usepackage{amssymb}
\usepackage{amsthm}
\usepackage{url}
\usepackage{hyperref}
\usepackage{fullpage}	

\newtheoremstyle{modified}
 {}
 {}
 {}
 {}
 {\bfseries}
 {}
 {.5em}
 {#1 #2.\thmnote{#3}}
\theoremstyle{modified}

\newtheorem{theorem}{Theorem}[section]
\newtheorem{lemma}[theorem]{Lemma}

\newcommand{\pr}{{\operatorname{pr}}}
\title{$S_n$-Equivariant Sheaves and Koszul Cohomology}
\author{David Yang}
\begin{document}
\maketitle
\begin{abstract}
We give a new interpretation of Koszul cohomology, which is equivalent under the Bridgeland-King-Reid equivalence to Voisin's Hilbert scheme interpretation in dimensions 1 and 2, but is different in higher dimensions. As an application, we prove that the dimension $K_{p,q}(B,L)$ is a polynomial in $d$ for $L=dA+P$ with $A$ ample and $d$ large enough.
\end{abstract}
\section{Introduction}
The Koszul cohomology of a line bundle $L$ on an algebraic variety $X$ was introduced by Green in \cite{Green}. Koszul cohomology is very closely related to the syzygies of the embedding defined by $L$ (if $L$ is very ample) and is thus related to a host of classical questions. We assume that the reader is already aware of these relations and the main theorems on Koszul cohomology; for a detailed discussion, the reader is referred to \cite{Green}.

Recently, in \cite{EL}, it was realized that a conjectural uniform asymptotic picture emerges as $L$ becomes more positive. We give a new interpretation of Koszul cohomology which we believe will clarify this picture. As an application, we prove that $K_{p,q}(B,dA+P)$ grows polynomially in $d$. In particular, this gives a partial answer to Problem 7.2 from \cite{EL}. More precisely, we establish

\begin{theorem}
Let $A$ and $P$ be ample line bundles on a smooth projective variety $X$. Let $B$ be a locally free sheaf on $X$. Then there is a polynomial that equals $\operatorname{dim} K_{p,q}(B,dA+P)$ for $d$ sufficiently large.
\end{theorem}

The inspiration for our new interpretation came from Voisin's papers \cite{Voisin1}, \cite{Voisin2} on generic Green's conjecture. To prove generic Green's conjecture, Voisin writes Koszul cohomology for $X$ a surface in terms of the sheaf cohomology of various sheaves on a Hilbert scheme of points of $X$. While this allows for the use of the geometry of Hilbert schemes in analyzing Koszul cohomology, this interpretation has multiple downsides.

The main downside is that it does not generalize well to higher dimensions, as the Hilbert scheme of points of $X$ is not necessarily smooth (or even irreducible) unless $X$ has dimension at most 2. It is unclear to us if this difficulty can be circumvented through uniform use of the curvillinear Hilbert scheme, but even if it can, the loss of properness creates various technical difficulties.

Our construction can be thought of as replacing the Hilbert scheme with a noncommutative resolution of singularities (though the word noncommutative need not ever be mentioned). Specifically, we replace the cohomology of sheaves on Hilbert schemes with the $S_n$-invariants of the cohomology of $S_n$-equivariant sheaves on $n$th fold products. This is motivated by the Bridgeland-King-Reid equivalence \cite{BKR}, which implies that for a curve or surface $X$, the derived category of coherent sheaves on the Hilbert scheme is equivalent to the derived category of $S_n$-equivariant coherent sheaves on $X^n.$ We note that this had been previously used to compute the cohomology of sheaves on the Hilbert scheme, for instance in \cite{Scala}.

We would also like to note two other results related to ours. If one chooses to work with ordinary sheaves, as opposed to $S_n$-equivariant sheaves, one recovers in effect a theorem implicit in \cite{Green2}, which was first stated explicitly in \cite{Inamdar}. We refer to the discussion after Theorem \ref{interpretation} for more detail. Finally, we note that \cite{gonality} proves a result which implies \ref{polynomialicity} for curves. The proof methods are very similar, and we are hopeful that Theorem \ref{interpretation} will help in extending their results to higher dimensions.

Section 2 of this paper is a short introduction to Koszul cohomology. The main goal is to describe the conjectural asymptotic story of \cite{EL}. We also describe (but do not prove) Voisin's interpretation of Koszul cohomology in terms of Hilbert schemes.

Section 3 contains our new interpretation. Section 4 then uses it to analyze the asymptotics of $K_{p,q}(dL+B)$. All the proofs in both sections are quite short.

This research was supported by an NSF funded REU at Emory University under the mentorship of David Zureick-Brown. We would also like to thank Ken Ono and Evan O'Dorney for advice and helpful conversations. Finally, we would like to thank Robert Lazarsfeld for his encouragement and correspondence. Both he and the referees gave many useful suggestions.

\section{Koszul Cohomology}
Our discussion of Koszul cohomology will be quite terse. For details and motivation, see \cite{Green} and \cite{EL}.

Let $X$ be a smooth projective algebraic variety of dimension $n$ over an algebraically closed field $k$ of characteristic zero and let $L$ be a line bundle on $X$. Form the graded ring $S=\operatorname{Sym}^{\bullet}H^0(L)$. For any coherent sheaf $B$ on $X$, we have a natural graded $S$-module structure on $M=\oplus_{n\geq 0}H^0(B+nL).$ From this we can construct the bigraded vector space $\operatorname{Tor}^{\bullet\bullet}_S(M,k).$ The $(p,q)$th Kozsul cohomology of $(B,L)$ is the $(p,p+q)$-bigraded part of this vector space. We will call its dimension the $(p,q)$ Betti number, and the two-dimensional table of these the Betti table.

In this paper, we will always take $B$ to be locally free.

The first theorem we describe is a duality theorem for Kozsul cohomology, proven in Green's original paper \cite{Green}. (In fact, Green proves a slightly stronger result.)

\begin{theorem}
\label{duality}
Assume that $L$ is base-point-free, and $H^i(B\otimes(q-i)L)=H^i(B\otimes(q-i-1)L)=0$ for $i=1,2,\ldots, n-1.$ Then we have a natural isomorphism between $K_{p,q}(B,L)^*$ and $K_{h^0(L)-n-p,n+1-q}(B^*\otimes K_X,L)$.
\end{theorem}

To compute the Kozsul cohomology, we can use either a free resolution of $M$ or a free resolution of $k$. Taking the minimal free resolution of $M$ relates the Kozsul cohomology to the degrees in which the syzygies of $M$ lie. In particular, this description shows that $K_{p,q}$ is trivial for $q < 0.$ Combining this with Theorem \ref{duality} and some elementary facts on Castelnuovo-Mumford regularity, it is possible to use this to show that $K_{p,q}$ is trivial for $q > n+1$ and is nontrivial for $q=n+1$ only when $p$ is within a constant (depending only on $B$ and $X$) of $h^0(L)$.

Extending this, in \cite{EL}, Ein and Lazarsfeld study when the groups $K_{p,q}$ vanish for $L$ very positive and $1\leq q\leq n$. More precisely, let $L$ be of the form $P+dA$, where $P$ and $A$ are fixed divisors with $A$ ample. Then as $d$ changes, Ein and Lazarsfeld prove that $K_{p,q}$ is nonzero if $p>O(d^{q-1})$ and $p<H^0(L)-O(d^{n-1})$. They conjecture that, on the other hand, $K_{p,q}$ is trivial for $p<O(d^{q-1}).$ It is known that if $q \geq 2$, then $K_{p,q}$ is trivial for $p < O(d).$ In particular, Theorem \ref{polynomialicity} only has content when $q$ is $0$ or $1$.

For the purposes of computation, the syzygies of $M$ are often hard to work with directly, so instead we will compute using a free resolution of $k$. The natural free resolution is the Kozsul complex
$$\cdots\wedge^2H^0(L)\otimes S\rightarrow H^0(L)\otimes S\rightarrow S\rightarrow k\rightarrow 0.$$
Tensoring by $M$, we immediately arrive at the following lemma:
\begin{lemma}
\label{Koszul}
If $q \geq 1$, $K_{p,q}(B,L)$ is equal to the cohomology of the three term complex
\begin{multline*}
H^0(B+(q-1)L)\otimes\wedge^{p+1}H^0(L)\rightarrow H^0(B+qL)\otimes\wedge^pH^0(L)\\ \rightarrow H^0(B+(q+1)L)\otimes\wedge^{p-1}H^0(L).
\end{multline*}
On the other hand, for $q=0$, $K_{p,q}(B,L)$ is equal to the kernel of the map $$H^0(B)\otimes\wedge^pH^0(L)\rightarrow H^0(B+L)\otimes\wedge^{p-1}H^0(L).$$
\end{lemma}

Finally, we say a few words on Voisin's Hilbert scheme approach to Koszul cohomology (see \cite{Voisin1}, \cite{Voisin2}). Let $X$ be a smooth curve or surface. Then denote by $\operatorname{Hilb}^n(X)$ the $n$th Hilbert scheme of points of $X.$ We have a natural incidence subscheme $I_n$ in $X\times\operatorname{Hilb}^n(X).$ Let $p\colon I_n\rightarrow X$ and $q\colon I_n\rightarrow\operatorname{Hilb}^n(X).$
\begin{theorem}
\label{Voisin}
Let $L^{[n]}$ be $\wedge^n\pr_{2*}p^*L.$ Then for $q\geq 1$, we have
$$K_{p,q}(B,L)\cong\operatorname{coker}(H^0(B+(q-1)L)\otimes H^0(L^{[p+1]})\rightarrow H^0((B+(q-1)L)\boxtimes L^{[p+1]}|_{I_{p+1}})).$$ Furthermore,
$$K_{p,0}\cong\operatorname{ker}(H^0(B)\otimes H^0(L^{[p+1]})\rightarrow H^0(B\boxtimes L^{[p+1]}|_{I_{p+1}}).$$
\end{theorem}
In fact, we have $H^0(L^{[n]})=\wedge^nH^0(L)$ and $H^0((B+(q-1)L)\boxtimes L^{[n]}|_{I_n})=\operatorname{ker}(H^0(B+qL)\otimes H^0(L^{[n-1]})\rightarrow H^0(B+(q+1)L)\otimes H^0(L^{[n-2]})).$ We thus see that Theorem \ref{Voisin} is a ``geometrized" version of Theorem \ref{Koszul}.

We would like to comment that one of our original motivations was to rewrite this in terms of the $q$th or $(q-1)$th sheaf cohomology of some sheaf. This is accomplished by the formalism of the next section combined with the results of \cite{Scala}.

\section{$S_n$-Equivariant Sheaves on $X^n$}In this section, we prove an analogue of Theorem \ref{Voisin}, but with the category of coherent sheaves on the Hilbert scheme replaced by the category of $S_n$-equivariant sheaves on $X^n$. Here we take $X$ to be smooth projective, $X^n$ to be its $n$th Cartesian power, and $B$ to be a locally free sheaf on $X$. 

An $S_n$-equivariant sheaf on $X^n$ is a sheaf on $X^n$ together with $S_n$-linearization in the sense of \cite[Chapter 5]{chrissG:geometricRepresentationTheory}. Note that for any map $X^n\rightarrow Y$ that is fixed under the $S_n$ action on $X^n$, we have an ``$S_n$-invariant pushforward" functor from the category of $S_n$-equivariant quasicoherent sheaves on $X^n$ to the category of quasicoherent sheaves on $Y$. The usual pushforward is equipped with a natural $S_n$ action, and the $S_n$-invariant pushforward is defined just to be the invariants under this action. The $S_n$-invariant cohomology $H^i_{S_n}$  is defined to be the derived functor of $S_n$-invariant pushforward to $\operatorname{Spec} k$.

We start by defining $L^{[n]}$ to be the sheaf $L\boxtimes \cdots \boxtimes L$ on $X^n$ with $S_n$ acting via the alternating action. By definition, we have $H^{\bullet}_{S_n}(L^{[n]})=\wedge^n H^{\bullet}(L).$ In the next section, we will also use the sheaf $D_L$, defined to be the sheaf $L\boxtimes\cdots \boxtimes L$ with $S_n$ acting via the trivial action. From the definitions, it is clear that $L^{[n]}=\mathcal{O}^{[n]}\otimes D_L.$

Let $\Delta_i$ be the subvariety of $X\times X^n$ defined as the set of points $(x,x_1,\ldots, x_n)$ with $x=x_i.$ Let $Z$ be the scheme theoretic union of the $\Delta_i$. Then $Z$ is equipped with two projections $p\colon Z\rightarrow X$ and $q\colon Z\rightarrow X^n.$ We will make crucial use of a long exact sequence
$$0\rightarrow\mathcal{O}_Z\rightarrow\oplus_{i_1}\mathcal{O}_{\Delta_i}\rightarrow\oplus_{i_1<i_2}\mathcal{O}_{\Delta_{i_1,i_2}}\rightarrow\cdots\rightarrow\mathcal{O}_{\Delta_{1,2,\ldots,n}}.$$
We will prove the following theorem.
\begin{theorem}
\label{interpretation}
Assume that $H^i(L)=H^i(B+mL)=0$ for all $i,m>0.$ Then for $q > 1$, $K_{p,q}(B,L)\cong H^{q-1}_{S_{p+q}}(p^*B\otimes q^*L^{[p+q]}).$ We also have an exact sequence
\begin{multline*}
0\rightarrow K_{p+1,0}(B,L)\rightarrow H^0(B)\otimes H^0_{S_{p+1}}(L^{[p+1]})\rightarrow H^0_{S_{p+1}}(p^*B\otimes q^*L^{[p+1]})\\ \rightarrow K_{p,1}(B,L)\rightarrow 0.
\end{multline*}
\end{theorem}
Note that when $B$ has no higher cohomology, we can write this more aesthetically as $K_{p,q}(B,L)=H^q_{S_{p+q}}((B\boxtimes L^{[p+q]})\otimes \mathcal{I}_Z),$ where $\mathcal{I}_Z$ is the ideal sheaf of $Z$.

Before proving the theorem, we need an exact sequence. Define $\Delta_{i_1,i_2,\ldots,i_m}$ to be the subvariety of $X\times X^n$ such that $(x,x_1,\ldots,x_n)$ is a point of $\Delta_{i_1,i_2,\ldots,i_m}$ if and only if $x=x_{i_1}=\cdots=x_{i_m}.$ Now note that we have a long exact sequence of $S_n$-equivariant sheaves on $X\times X^n$
$$0\rightarrow\mathcal{O}_Z\rightarrow\oplus_{i_1}\mathcal{O}_{\Delta_i}\rightarrow\oplus_{i_1<i_2}\mathcal{O}_{\Delta_{i_1,i_2}}\rightarrow\cdots\rightarrow\mathcal{O}_{\Delta_{1,2,\ldots,n}}.$$

We define the map from $\mathcal{O}_{\Delta_{i_1,i_2,\cdots, i_m}}$ to $\mathcal{O}_{\Delta_{j_1,j_2,\cdots j_{m+1}}}$ to be nonzero if and only if $\Delta_{j_1,\cdots, j_{m+1}}$ is a subvariety of $\Delta_{i_1,i_2,\cdots, i_m}$, in which case we define it to be the natural map induced by the inclusion, up to sign. If the difference between $\{i_1,\cdots i_m\}$ and $\{j_1,\cdots j_{m+1}\}$ is $j_k$, then we modify this map by a factor of $(-1)^{k-1}.$

We need to be careful to define our $S_n$-equivariant structure on our complex in a way compatible with this modification. For any element $\sigma\in S_n,$ we have $\sigma_*\mathcal{O}_{\Delta_{i_1,i_2,\cdots, i_m}}\cong\mathcal{O}_{\Delta_{h_1,h_2,\cdots,h_m}},$ where $h_1 < h_2 <\cdots < h_m$ and $\{h_1,\cdots,h_m\}=\{\sigma{i_1},\cdots, \sigma{i_m}\}$. This gives each term of our complex a natural $S_n$-action. To make our maps $S_n$-equivariant, we modify the action by the sign of the permutation sending $h_1,h_2,\cdots h_m$ to $\sigma{i_1},\cdots \sigma{i_m}.$

After these definitions, it is quite easy to check that we have indeed defined a $S_n$-equivariant complex. Exactness is harder; a proof can be found in Appendix A of \cite{Scala}.

We now prove Theorem \ref{interpretation}.
\begin{proof}
Tensor our exact sequence by $B\boxtimes L^{[n]}$. Each term of the complex is of the form $\oplus_{i_1<i_2<\cdots<i_m}B\boxtimes L^{[n]}|_{\Delta_{i_1,i_2,\ldots,i_m}}.$ We compute the $S_n$-invariant cohomology of this sheaf. This is the same as the $S_m\times S_{n-m}$ invariant cohomology of $B\boxtimes L^{[n]}|_{\Delta_{1,2,\ldots,m}}.$ Now note that there is a natural $S_{n-m}$-equivariant isomorphism between $\Delta_{1,2,\ldots, m}$ and $X \times X^{n-m}.$ This sends $B\boxtimes L^{[n]}|_{\Delta_{1,2,\ldots, m}}$ to $(B+mL)\boxtimes L^{[n-m]}.$ The action of $S_m$ on $H^0(B+mL)$ is trivial, as the alternating actions coming from our modified $S_n$-equivariance and $L^{[n]}$ cancel, and thus we can simply consider the $S_{n-m}$-invariant cohomology. By our assumptions and the fact that $H^{\bullet}_{S_n}(L^{[n]})=\wedge^nH^{\bullet}(L),$ we see that the higher $S_{n-m}$-invariant cohomology of $B\boxtimes L ^{[n]}|_{\Delta_{1,2,\ldots,m}}$ vanishes and that $H^0_{S_{n-m}}(B\boxtimes L^{[n]}|_{\Delta_{1,2,\ldots,m}})\cong H^0(B+mL)\otimes\wedge^{n-m}H^0(L).$

Our long exact sequence now immediately shows that the $S_n$-invariant cohomology of $B\boxtimes L^{[n]}|_Z$ is given by the cohomology of the complex
$$0\rightarrow H^0(B+L)\otimes\wedge^{n-1}H^0(L)\rightarrow H^0(B+2L)\otimes\wedge^{n-2}H^0(L)\rightarrow\cdots$$
and, setting $n=p+q$, our result immediately follows from Lemma \ref{Koszul}.
\end{proof}

We note that in particular, if $H^{q-1}(p^*B\otimes q^*L^{[p+q]})=0$ (note that this is non-invariant cohomology!), then $K_{p,q}(B,L)$ vanishes. A very similar statement was first claimed in the proof of Theorem 3.2 of \cite{Green2}, with a minor mistake corrected by Inamdar in \cite{Inamdar}. Since then, this statement has been used in various places (e.g., \cite{HwangTo}\cite{LazarsfeldPareschiPopa}.)

Note that for any map $X^n\rightarrow Y$ that is fixed under the $S_n$ action on $X^n$, we have a ``$S_n$-invariant pushforward" functor from the category of $S_n$-equivariant quasicoherent sheaves on $X^n$ to the category of quasicoherent sheaves on $Y$. The usual pushforward is equipped with a natural $S_n$ action, and the $S_n$-invariant pushforward is defined just to be the invariants under this action. The $S_n$-invariant cohomology is defined to be the $S_n$-invariant pushforward to $\operatorname{Spec} k$
\section{Polynomial Growth of $K_{p,q}(B,L_d)$}
In this section, we will let $L_d=dA+P,$ where $A$ and $P$ are lines bundles with $A$ ample. We allow $B$ to be any locally free sheaf on $X$. Our main theorem is the following.
\begin{theorem}
\label{polynomialicity}
For $d$ sufficiently large, $\operatorname{dim} K_{p,q}(B,L_d)$ is a polynomial in $d.$
\end{theorem}
\begin{proof}
We start by showing that $K_{p,q}(B,L_d)$ is trivial for $q \geq 2$ and $d$ large enough. This has been known at least since \cite{EL2}, but we will reprove it here for the sake of self-containedness. Let $\mathcal{F}$ be the sheaf $\pr_{2*}\mathcal{O}_Z\otimes P^{[p+q]}.$ As $q$ is finite, pushforward along it is exact. Therefore, by Theorem \ref{interpretation}, we know that $$K_{p,q}(B,L_d)\cong H^{q-1}_{S_{p+q}}(\mathcal{F}\otimes dD_{A})$$ (see the beginning of the previous section for the definition of $D_A$). But $D_{A}$ is ample, so this is zero for large enough $d.$

By Theorem \ref{interpretation}, it suffices to show that the dimensions of $K_{p+1,0}(B,L_d), H^0(B)\otimes H^0_{S_{p+1}}(L_d^{[p+1]}),$ and $H^0_{S_{p+1}}(p^*B\otimes q^*L_d^{[p+1]})$ all grow polynomially in $d.$ Using that $H^{\bullet}_{S_{n}}(L_d^{[n]})\cong\wedge^nH^{\bullet}(L_d),$ we immediately see that $\operatorname{dim} H^0(B)\otimes H^0_{S_{p+1}}(L_d^{[p+1]})$ grows polynomially in $d.$

It follows immediately from the second part of Theorem \ref{interpretation} that $K_{p+1,0}(B,L_d)\cong H^0_{S_{p+1}}(B\boxtimes L_d^{[p+1]}\otimes\mathcal{I}_Z).$ Letting $\mathcal{G}$ denote the sheaf $\pr_{2*}(B\boxtimes L_d^{[p+1]}\otimes\mathcal{I}_Z),$ we see that $$K_{p+1,0}\cong H^0_{S_{p+1}}(\mathcal{G}\otimes dD_{A}).$$

To proceed, we take the $S_{p+1}$-invariant pushforward of this to $\operatorname{Sym}^{p+1}(X).$ Let the $S_{p+1}$-invariant pushforward of $\mathcal{G}$ be $\mathcal{H}.$ Note that $D_{A}$ is the pullback (with the natural $S_{p+1}$-action) of an ample line bundle $A'$ on $\operatorname{Sym}^{p+1}(X).$ We thus see that $K_{p+1,0}(B,L_d) =H^0(\mathcal{H}\otimes dA'),$ and by the existence of the Hilbert polynomial, its dimension is a polynomial for large enough $d.$ An identical argument shows that $\operatorname{dim}H^0_{S_{p+1}}(p^*B\otimes q^*L_d^{[p+1]})$ is a polynomial for large enough $d.$
\end{proof}

Note that this theorem answers the first part of Problem 7.2 in \cite{EL} in the affirmative. Their specific question asks if the triviality of a specific Koszul cohomology group is independent of $d$ if $d$ is large. As any nonzero polynomial has only finitely many roots, Theorem \ref{polynomialicity} implies that the answer is yes.

\bibliographystyle{plain}
\bibliography{polygraphsyzygies}
\end{document}